\documentclass[12pt,leqno,twoside]{amsart}
\usepackage{amssymb}
\usepackage{latexsym}

\usepackage{epsfig}
\usepackage{srcltx}

\usepackage{amsfonts,color, soul}

\usepackage[colorlinks=true, pdfstartview=FitV, linkcolor=blue, citecolor=blue, urlcolor=blue]{hyperref}

\usepackage{graphicx,enumitem}

\newtheorem{theorem}{Theorem}[section]

\newtheorem{lemma}[theorem]{Lemma}

\newtheorem{conjecture}[theorem]{Conjecture}
\theoremstyle{remark}

\theoremstyle{remark}

\theoremstyle{definition}

\theoremstyle{remark}

\theoremstyle{remark}

\theoremstyle{remark}

\renewcommand{\int}{{\mathrm{int}}}

\newcommand{\pol}{{\mathrm{pol}}}

\newcommand{\Sing}{{\mathrm{Sing\hspace{2pt}}}}

\renewcommand{\j}{{\mathrm{j}}}

\newcommand{\grad}{\mathop{\mathrm{grad}}\nolimits}
\newcommand{\ord}{\mathop{\mathrm{ord}}\nolimits}

\newcommand{\rk}{\mathop{\mathrm{rank}}\nolimits}
\newcommand{\corank}{\mathop{\mathrm{corank}}\nolimits}

\newcommand{\m}{\setminus}

\newcommand{\bC}{{\mathbb C}}

\newcommand{\bP}{{\mathbb P}}

\newcommand{\bN}{{\mathbb N}}

\begin{document}

\title[Polar degree   and Huh's conjectures]
 {On Huh's conjectures for the polar degree}

\author{\sc Dirk Siersma}  

\address{Institute of Mathematics, Utrecht University, PO
Box 80010, \ 3508 TA Utrecht, The Netherlands.}

\email{D.Siersma@uu.nl}

\author{Joseph Steenbrink}
\address{IMAPP, Radboud University Nijmegen, Nijmegen, The Netherlands}

\email{j.steenbrink@math.ru.nl}

\author{\sc Mihai Tib\u ar }

\address{Univ. Lille, CNRS, UMR 8524 - Laboratoire Paul Painlev\'e, F-59000 Lille, France.}

\email{mihai-marius.tibar@univ-lille.fr}

\thanks{DS and MT express their gratitude to the Mathematisches Forschungsinstitut Obewolfach for supporting this research project through the Research in Pairs program 2017, and  acknowledge the support of the Labex CEMPI (ANR-11-LABX-0007-01). }




\dedicatory{}



\begin{abstract}
 We  prove a precise version of  a general conjecture on the polar degree stated by June Huh.
We confirm Huh's conjectural list of all projective hypersurfaces with  isolated singularities and polar degree equal to 2.
\end{abstract}

\maketitle

\setcounter{section}{0}

\section{Introduction}\label{s:intro}


Let $f \in \bC[x_0,\ldots,x_n]$, $n\ge 2$,  be a homogeneous polynomial of degree $d \geq 2$ and $V := \{f=0\}\subset \mathbb{P}^n$.
The \emph{polar degree} of $f$ is defined as the topological degree of the gradient mapping
\begin{equation}\label{eq:grad}
\grad f : \bP^{n}\m \Sing V \to  \bP^{n}.
\end{equation}
It only depends on the reduced structure of $V$ and thus can be denoted by $\pol(V)$. 
 This fact has been conjectured by Dolgachev \cite{Do} and proved by Dimca and Papadima \cite{DP} as a consequence of the following interpretation:
\begin{equation}\label{eq:pollink}
 \pol(V) = \rk H_{n-1}(V \m H)
\end{equation}
where $H\subset \bP^{n}$ is a general hyperplane.
More than 160 years ago, Otto Hesse had claimed \cite{Hes51, Hes59} that $\pol(V)=0$ if and only if the hypersurface $V$ is a cone\footnote{Gordan and Noether  \cite{GN76} showed that this claim is true modulo a birational transformation.}, but this has been
disproved by a nice example due to Gordan and Noether \cite{GN76}, see \cite{Huh} for a discussion of this example and other remarks.  More recently
Dolgachev \cite{Do} considered  hypersurfaces with $\pol(V) =1$ (so-called \emph{homaloidal}) and classified such plane curves, cf the list below, see also \cite{Di2, FM}.  We refer to \cite{Huh} for some recent papers about various questions around the polar degree. 

In this note we focus on the polar degree of hypersurfaces with isolated singularities. The formula:
\begin{equation}\label{eq:polmu}
  \pol(V) = (d-1)^{n} - \sum_{p\in \Sing V} \mu(V, p)
\end{equation}
proved by Dimca and Papadima \cite{DP} in terms of the degree $d$ and the Milnor numbers of $V$ at its singular points, implies in particular that the quadratic hypersurface  is the only smooth $V$ which is homaloidal. In order to bound $\pol(V)$ from below,  Huh \cite{Huh} used the theory of slicing by \emph{pencils with isolated singularities in the axis} introduced and developed in \cite{Ti-book}. He found  a key inequality in a more general setting relating $\pol (V)$ to the sectional Milnor number $\mu^{\left< n-2\right> }(V,p)$ of  any chosen singular point $p\in \Sing V$ of a hypersurface $V$ with isolated singularities and\footnote{it is explicitly stated in \emph{loc.cit.} that the inequality does not apply to $V$ which is a cone with apex at $p$; indeed, being a cone implies $\pol(V)=0$.} $\pol(V) >0$, that we shall call \emph{Huh's inequality}: 
\begin{equation}\label{eq:pol}
\pol(V) \ge  \mu^{\left< n-2\right> }(V, p) .
\end{equation}
With this bound at hand, Huh could prove   Dimca and Papadima's \cite{DP} conjectural list of all homaloidal hypersurfaces with isolated singularities \cite[Theorem 4]{Huh}:

{\it A projective hypersurface $V\subset \bP^{n}$ with only isolated singularities has polar degree 1 if and only if 
it is one of the following, after a linear change of homogeneous coordinates:}
\begin{enumerate}
\rm\item[(i)]{\it \hspace{1cm} $(n\ge 2, d=2)$  a smooth quadric:
\[  f = x_{0}^{2} + \cdots + x_{n}^{2} = 0.
\]
}
\rm\item[(ii)]{\it \hspace{1cm} $(n= 2, d=3)$ the union of three non-concurrent lines:
\[  f = x_{0}x_{1}x_{2} = 0, \hspace{.6cm} (3A_{1}).
\]
}
\rm\item[(iii)]{\it \hspace{1cm} $(n= 2, d=3)$ the union of a smooth conic and one of its tangents:
\[  f = x_{0}(x_{1}^{2} + x_{0}x_{2}) =0, \hspace{.6cm} (A_{3}).
\]
}
\end{enumerate}

This list contains the  plane curves found by Dolgachev \cite{Do}.

\medskip
June Huh conjectured a general finiteness principle for 
projective hypersurfaces with isolated singularities and fixed polar number  $\pol(V) =k$.  In case $k=2$, he conjectured that
such hypersurfaces  are only those in the following list:

\begin{conjecture}\rm \cite[Conjecture 20]{Huh} \it
 A projective hypersurface $V\subset \bP^{n}$ with only isolated singularities has polar degree 2 if and only if 
it is one of the following, after a linear change of homogeneous coordinates:
\begin{enumerate}
\rm\item{\it \hspace{.6cm} $(n=3, d=3)$  a normal cubic surface containing a single line:
\[  f = x_{0} x_{1}^{2} + x_{1} x_{2}^{2} + x_{1} x_{3}^{2} + x_{2}^{3} = 0,   \hspace{.6cm} (E_{6}).
\]
}
\rm\item{\it \hspace{.6cm} $(n=3, d=3)$  a normal cubic surface containing two lines:
\[  f =  x_{0} x_{1}x_{2} + x_{0} x_{3}^{2}  + x_{1}^{3} = 0, \hspace{.6cm} (A_{5}, A_{1}).
\]
}
\rm\item{\it \hspace{.6cm} $(n=3, d=3)$  a normal cubic surface containing three lines and three binodes:
\[  f =  x_{0} x_{1}x_{2}  + x_{3}^{3} = 0, \hspace{.6cm} (A_{2}, A_{2}, A_{2}).
\]
}
\rm\item{\it \hspace{.6cm} $(n=2, d=5)$  two smooth conics meeting at a single point and their common tangent:
\[  f =  x_{0} (x_{1}^{2}  + x_{0}x_{2}) (x_{1}^{2}  + x_{0}x_{2}+ x_{0}^{2}) = 0, \hspace{.6cm} (J_{2,4}).
\]
}
\rm\item{\it \hspace{.6cm} $(n=2, d=4)$  two smooth conics meeting at a single point:
\[  f =  (x_{1}^{2}  + x_{0}x_{2}) (x_{1}^{2}  + x_{0}x_{2}+ x_{0}^{2}) = 0, \hspace{.6cm} (A_{7}).
\]
}
\rm\item{\it \hspace{.6cm} $(n=2, d=4)$  a smooth conic, a tangent and a line passing  through the tangency point:
\[  f = x_{0}(x_{0}+ x_{1}) (x_{1}^{2}  + x_{0}x_{2})  = 0, \hspace{.6cm} (D_{6}, A_{1}).
\]
}
\rm\item{\it \hspace{.6cm} $(n=2, d=4)$  a smooth conic  and two tangent  lines:
\[  f = x_{0}x_{2} (x_{1}^{2}  + x_{0}x_{2})  = 0, \hspace{.6cm} (A_{1}, A_{3},A_{3}).
\]
}
\rm\item{\it \hspace{.6cm} $(n=2, d=4)$  three concurrent lines and a line not meeting the center point:
\[  f = x_{0}x_{1} x_{2}(x_{0}+ x_{1})  = 0, \hspace{.6cm} (D_{4}, A_{1}, A_{1},A_{1}).
\]
}
\rm\item{\it \hspace{.6cm} $(n=2, d=4)$  a cuspidal cubic and its tangent at the cusp:
\[  f = x_{0}(x_{1}^{3}+ x_{0}^{2}x_{2})  = 0, \hspace{.6cm} (E_{7}).
\]
}
\rm\item{\it \hspace{.6cm} $(n=2, d=4)$  a cuspidal cubic and its tangent at the smooth flex point:
\[  f = x_{2}(x_{1}^{3}+ x_{0}^{2}x_{2})  = 0, \hspace{.6cm} (A_{2}, A_{5}).
\]
}
\rm\item{\it \hspace{.6cm} $(n=2, d=3)$  a cuspidal cubic:
\[  f = x_{1}^{3}+ x_{0}^{2}x_{2}  = 0, \hspace{.6cm} (A_{2}).
\]
}
\rm\item{\it \hspace{.6cm} $(n=2, d=3)$  a smooth conic and a secant line:
\[  f = x_{1}(x_{1}^{2}+ x_{0}x_{2})  = 0, \hspace{.6cm} (A_{1}, A_{1}).
\]
}

\end{enumerate}
\end{conjecture}

Huh showed his conjecture for plane curves, cubic and quartic surfaces \cite{Huh}.
Fasarella and Medeiros \cite{FM} classified plane curves with $\pol(V) =2$ as a consequence of  formula 
\eqref{eq:pollink} and the computation of Euler characteristics.  
\medskip

We first prove Huh's conjecture for $\pol(V) =2$, then Huh's finiteness conjecture for $\pol(V) =k$. More precisely:

\begin{theorem}\label{t:mainPol2}
 The hypersurfaces $V\in \bP^{n}$ with isolated singularities and of polar degree 2 are only those in Huh's list.
 
In particular there are no such hypersurfaces  for $n>3$.
\end{theorem}

\medskip


The main idea is to consider the affine hypersurface $V\setminus H$, for some general hyperplane $H$ not passing through any singularity, and to deform it into the  cone $V_{n,d} \setminus H$ defined by $f_{n, d}:= x_{1}^{d}+\cdots + x_{n}^{d} =0$.
This induces an embedding of the direct sums of the Milnor lattices of the singularities of $V$ into the Milnor lattice of the isolated singularity of $f_{n, d}$. We apply to this deformation the semicontinuity of the spectrum proved by Varchenko \cite{Va} and  Steenbrink \cite{St85} (Section \ref{polar2n>3}).

\smallskip
 June Huh formulated  the following general principle for any $k>2$:

\begin{conjecture}\label{c:huh2}\rm \cite[page  1545]{Huh} \it
There is no projective hypersurface $V\subset \bP^{n}$ of polar
 degree $k$ with only isolated singular points, for sufficiently 
large $n$ and $d = \deg V$.
\end{conjecture}

We answer positively  to Huh's general conjecture in the following precise way:
\begin{theorem}[Finiteness Theorem]\label{t:Huh-compact}
For any integer $k\geq 2$, let $K_k$ denote the set of pairs of integers $(n,d)$ with $n\geq 2$ and $d\geq 3$, such that there exists a projective hypersurface $V$ in $\mathbb{P}^n$ of degree $d$ with isolated singularities and  $\pol(V)=k$.

 Then $K_k$ is finite for any $k\geq 2$. 
 \end{theorem}
 
By Theorem \ref{t:mainPol2} we are showing that $K_2=\{(2,3),(2,4),(2,5),(3,3)\}$ and that this is a sharp equality. 
To prove Theorem  \ref{t:Huh-compact} for $k\ge 3$, we show that any  $(n,d) \in K_k$ satisfies the inequalities  $n < \max(k,5+\log_2 k)$ by Theorem \ref{t:h0},  and $d < \max \{ 2+ \ell_{n,k},  (n+\ell_{n,k})(k+2)/(n-1)\}$
 by Theorem \ref{t:Huh2}, where  $\ell_{n,k} := \min \left\{\ell \in \bN \ \middle| {n+\ell \choose n}>k\right\}$.




\section{Spectrum and the semicontinuity argument}\label{spectrum}

We recall here the properties of the spectrum that we shall use in the proofs.

\subsection{Properties}
Let $X$ be a complex manifold of dimension $n$ and $f:X \to \mathbb{C}$ a holomorphic function. 
Let $x\in X$ be an isolated critical point of $f$ with Milnor fibre $X_{f,x}$. The group $\tilde{H}^{n-1}(X_{f,x})$ is free abelian of finite rank equal to the Milnor number $\mu(f,x)$. Let $T$ denote the monodromy operator and let $T_s$ be its semisimple part. Then $\tilde{H}^{n-1}(X_{f,x})$ underlies a mixed Hodge structure with Hodge filtration $F^\bullet = (0 \subset F^{n-1} \subset F^{n-2} \subset \cdots \subset F^0=\tilde{H}^{n-1}(X_{f,x},\mathbb{C}))$, preserved by $T_s$. By the monodromy theorem, the eigenvalues of $T_s$ are roots of unity. 


The \emph{spectrum} $\mathrm{Sp}(f,x)$ is an element of the group ring $\mathbb{Z}[\mathbb{Q}]$:
\[ \mathrm{Sp}(f,x) = \sum_\alpha n_\alpha(\alpha)\]
a finite sum with $n_\alpha\in \mathbb{N}$. Here $n_\alpha$ is defined as the multiplicity of $\exp(-2\pi i\alpha)$ as an eigenvalue of $T_s$ acting on $F^p/F^{p+1}$, where $p=\lfloor n-1-\alpha \rfloor$. The spectrum has the following properties.
\begin{description}
\item[Range] $n_\alpha \neq 0 \Rightarrow -1<\alpha<n-1$.
\item[Symmetry] If $\alpha+\beta=n-2$ then $n_\alpha=n_\beta$. See \cite[Corollary to Lemma 13.14]{AGV2} for these two properties. 
\item[Stability] If $g:X \times \mathbb{C}$ is defined by $g(z,t)=f(z)+t^2$, and $x\in X$ with $\mathrm{Sp}(f,x) = \sum_\alpha n_\alpha(\alpha)$, then $\mathrm{Sp}(g,(x,0)) = \sum_\alpha n_\alpha(\alpha+\frac 12)$. See  \cite[Corollary 1 to Theorem 13.7]{AGV2}.

\item[$\mu$-constant deformation invariance]   The spectrum is constant in any deformation of isolated hypersurface singularity germs with constant Milnor number,  see \cite{St85}.
\item[Semicontinuity] 
Let $(Y_{0}, 0)$, where  $Y_{0}:= \{ g=0\} \subset \bC^{n}$,  be the germ  of a hypersurface with isolated singularity.  Let $g_{s}$ be a good representative of a deformation of $g$ such that $Y_{s} := \{ g_{s}=0\}$ has isolated singular points 
$\mathrm{Sing}(Y_{s})=\{p_1,\ldots,p_r\}$ which tend to the origin  0 when $s\to 0$.
Let  $g_i:(\bC^{n},p_i)\to (\mathbb{C},0)$ be a local equation for $Y_{s}$ at $p_i$, for $i=1,\ldots,r$. 
Then for each $a\in \mathbb{R}$ one has:
\[
\sum_{i=1}^r \deg_{]a,a+1[}\mathrm{Sp}(g_i,p_i) \leq \deg_{]a,a+1[} \mathrm{Sp}(g,0)
\]
in case we have a deformation of lower weight of a quasi-homogeneous function germ $g$,
and
\[
\sum_{i=1}^r \deg_{]a,a+1]}\mathrm{Sp}(g_i,p_i) \leq \deg_{]a,a+1]} \mathrm{Sp}(g,0)
\]
in general.

Here for any $A \subset \mathbb{R}$ the function $\deg_A:\mathbb{Z}[\mathbb{Q}]\to \mathbb{Z}$ is defined by 
$\deg_A(\sum_\alpha n_\alpha(\alpha)):=\sum_{\alpha\in A} n_\alpha$. The first inequality  is a special case of a result of Varchenko \cite{Va}, and the second inequality follows from \cite[Theorem 2.4]{St85}.
\end{description}

\subsection{Deformation to $f_{n,d}$}
Let $H:= \{x_0 = 0\} \subset \bP^{n}$ be a generic hyperplane with respect to our hypersurface $V := \{f=0\}$
 and write $f = f_d + x_0 f_{d-1} + \cdots + x_0^d f_0$
 where $f_i\in\mathbb{C}[x_1,\ldots,x_n]$ is homogeneous of degree $i$.
We consider the 1-parameter 
 family of polynomials 
$g_s (1,x_1, \cdots , x_n) := f_d + s f_{d-1} + \cdots + s^d f_0$ on $\bC^{n} = \bP^{n} \m H$, see e.g. \cite{Br}.
Then $f_{d}$ is a general homogeneous polynomial which, as germ at the origin $0 \in \bC^{n}$,
 is topological equivalent to  the polynomial $f_{n,d}:=\sum_{i=1}^{n} x_i^d$. 
 
 Therefore the family $g_s$ describes a deformation of lower weight of $g_{1} = f_{|\bC^{n}}$ to $g_{0} = f_{d}$, hence a deformation from the hypersurface $V\m H$   to the hypersurface  $\{f_{d}=0\}$, which is topological equivalent to the hypersurface  $\{ f_{n,d}=0\}$.
 We may thus  apply  the semicontinuity of the spectrum and,  if the local singularities $(V,p_i)$ of $V$ are defined by $f_i = 0$, we get:

\begin{equation}\label{eq:1}
\sum_{i=1}^r \deg_{]a,a+1[}\mathrm{Sp}(f_i,p_i) \leq \deg_{]a,a+1[} \mathrm{Sp}(f_{n,d},0)
\end{equation}
and
\begin{equation}\label{eq:2}
\sum_{i=1}^r \deg_{]a,a+1]}\mathrm{Sp}(f_i,p_i) \leq \deg_{]a,a+1]} \mathrm{Sp}(f_{n,d},0).
\end{equation}

\subsection{Spectra of special singularities}

As Huh showed in \cite[Lemma 19]{Huh} by using his bound \eqref{eq:pol} and Arnold's classification results \cite{AGV2},   the singularities which may occur on a projective hypersurface of dimension $n-1$ of polar degree $1$ or $2$ are within the following list:
\medskip
\noindent
\begin{center}
\begin{tabular}{|l|l|c|}
\hline
type & equation & Milnor number \cr
\hline
$A_k$ with $k\geq 1$ & $x_1^{k+1}+q_{n-1}$ & $k$\cr
$D_k$ with $k\geq 4$ & $x_1^2x_2+x_2^{k-1} +q_{n-2}$ & $k$ \cr
$E_{6k}$ with $k\geq 1$  & $x_1^3+x_2^{3k+1}+q_{n-2}$ & $6k$ \cr
$E_{6k+1}$ with $k\geq 1$  & $x_1^3+x_1x_2^{2k+1}+q_{n-2}$ & $6k+1$\cr
$E_{6k+2}$ with $k\geq 1$ & $x_1^3+x_2^{3k+2}+q_{n-2}$ & $6k+2$\cr
$J_{k,i}$ with $k\geq 2,\ i\geq 0$ & $x_1^3+x_1^2x_2^k+x_2^{3k+i}+q_{n-2}$ & $6k-2+i$ \cr
\hline
\end{tabular}
\end{center}

\bigskip
\noindent
where $q_{n-1-j}:=x_{j+2}^2+\cdots + x_{n}^2$, and in each case we have exhibited one of the possible equations in the $\mu$-class of the singularity. Note that all types except $J_{k,i}$ with $i>0$  are represented by a weighted homogeneous function. It was shown in \cite[Example 5.11]{St76} for $f\in \mathbb{C}[x_1,x_2]$ with an isolated singular point at $0$ and weighted homogeneous with weights $w_1$ and $w_2$, that:
\[ \sum n_\alpha t^{\alpha+1} = \frac{t^{w_1}-t}{1-t^{w_1}}\cdot \frac{t^{w_2}-t}{1-t^{w_2}} \]
The weights are as follows:

\medskip
\noindent
\begin{center}
\begin{tabular}{|c|c|c|c|c|c|c|c|}
\hline
Type & $A_k$ & $D_k$ & $E_{6r}$ & $E_{6r+1}$ & $E_{6r+2}$ & $J_{k,0}$ \cr
\hline
$w_1$ & $\frac 1{k+1}$ & $\frac{k-2}{2k-2}$ & $\frac 13$ & $\frac 13$ & $\frac 13$ & $\frac 13$ \cr
\hline
$w_2$ & $\frac 12$ & $\frac 1{k-1}$ & $\frac 1{3r+1}$ & $\frac 2{6r+3}$ & $\frac 1{3r+2}$ & $\frac 1{3k}$ \cr
\hline
\end{tabular}
\end{center}

\bigskip
\noindent
and the spectra of the corresponding curve singularities are:

\medskip
\noindent
\begin{center}
\begin{tabular}{|l|l|}
\hline
$A_k$ & $\sum_{j=1}^k\left(-\frac 12+\frac j{k+1}\right)$ \cr
$ D_k$ & $(0)+\sum_{j=1}^{k-1}\left(-\frac 12 + \frac{2j-1}{2k-2}\right)$ \cr
$ E_{6k}$ & $ \sum_{j=1}^{3k}\left(-\frac 23+\frac j{3k+1}\right) +  \sum_{j=1}^{3k}\left(-\frac 13+\frac j{3k+1}\right)$ \cr
$ E_{6k+1}$ & $ (0) + \sum_{i=1}^2\sum_{j=1}^{3k} \left( -\frac i3+\frac{2j}{6k+3}\right)$ \cr
$ E_{6k+2}$ & $ \sum_{j=1}^{3k+1}\left(-\frac 23+\frac j{3k+2}\right) +  \sum_{j=1}^{3k+1}\left(-\frac 13+\frac j{3k+2}\right)$ \cr
$J_{k,0}$ & $\sum_{j=1}^{3k-1}\left(-\frac 23+\frac j{3k}\right) +  \sum_{j=1}^{3k-1}\left(-\frac 13+\frac j{3k}\right).$ \cr
\hline
\end{tabular}
\end{center}

\bigskip
\noindent

If $i>0$ then the singularity $J_{k,i}$ is still nondegenerate with respect to its Newton diagram, so its spectrum can be computed as in \cite[Sect. 5.15]{St76}. See also \cite[Sect. 13.3.4]{AGV2}. To describe the result, we only  list the negative spectral numbers (again in the case of the corresponding curve singularities).  These numbers occur in two groups, those from the first group having  denominator $3k$ and those from the second having denominator $6k+2i$.
The numerators occurring depend on the parities of $k$ and $i$. The numerators of the first group are:

\medskip
\noindent
\begin{center}
\begin{tabular}{|l|l|}
\hline
$k$ even & $-2k+1,-2k+2,\ldots,-3k/2,-k+1,-k+2,\ldots,-1$ \cr
$k$ odd & $-2k+1,-2k+2,\ldots,(-3k-1)/2,-k+1,-k+2,\ldots,-1$\cr
\hline
\end{tabular}
\end{center}

\bigskip
\noindent
and the numerators of the second group are the integers $\ell$ in the interval  $(-(3k+i),0)$ with the same parity as $i$. 

Observe that the spectral numbers in the second group are all greater than $-\frac 12$. 
\subsection{Spectrum of $f_{n,d}$}\label{ss:gen}
The singularity $f_{n,d}=\sum_{i=1}^{n} x_i^d$   
has Milnor number $(d-1)^{n}$ and spectrum $\sum_{k\in \mathbb{Z}}u_k(\frac kd)$, where: 
\[u_k = \sharp\{(a_1,\ldots,a_{n})\in \mathbb{N}^{n}\mid 1\leq a_j\leq d-1 \mbox{ and }\sum_ja_j=k+d\}. \]



\section{Hypersurfaces of polar degree $2$}\label{polar2n>3}

\begin{theorem}\label{t:pol2}
There is no projective hypersurface $V\subset \bP^{n}$ with isolated singularities, of polar degree $\pol(V)=2$, and dimension and degree at least three (i.e. $n\ge 4$ and $d\ge 3$). 
\end{theorem}

\begin{proof}
Suppose that there exists a hypersurface $V\subset \bP^n$ with isolated singularities
of degree $d$ and polar degree 2. By \cite[Lemma 19]{Huh} we know that the singularities of $V$ are of type $A,D,E,J_{*,*}$. It follows that the singularity $f_{n,d}$ deforms to those singularities in one fibre, whose Milnor numbers add up to $(d-1)^{n}-2$, which is just two less than the Milnor number of $f_{n,d}$.
The singularity $f_{n,d}$ has as its two smallest spectral numbers $-1+\frac{n}d$ with multiplicity 1, and $-1+\frac{n+1}d$ with multiplicity $n$.  

\begin{lemma}\label{l:1}
 Suppose $d\geq 3$ and $n\geq 3$.  Then either $n=3$ and $d\le 4$, or
 $d=3$ and $4\le n \le 5$. 
 \end{lemma}
\begin{proof} All singularities of type $A,D,E,J_{*,*}$ in dimension $n-1$ have a smallest spectral number which is greater than $-\frac 23+\frac{n-2}2=\frac {n-1}2-\frac 76$.  Suppose that the second spectral number $\frac{n+1}d-1$ of $f_{n,d}$  is smaller or equal to $t_{n}:= \frac{n-1}2-\frac 76$. As this spectral number has multiplicity equal to $n$, it follows that $\deg_{]-\infty,t_{n}]}\mathrm{Sp}(f_{n,d})\geq n+1$. Therefore:
\[ \deg_{]t_{n},\infty[}\mathrm{Sp}(f_{n,d}) < (d-1)^n-n-1<(d-1)^n-2=\sum_i \deg_{]t_{n},\infty[}\mathrm{Sp}(f_i,p_i)
\] 
since $n\ge 2$, and the last equality is due to our assumption $\pol(V) =2$. 
This contradicts the semicontinuity  of the spectrum \eqref{eq:2}.

So the second spectral number $\frac{n+1}d-1$ of $f_{n,d}$ must be greater than  $\frac {n-1}2-\frac 76$. But from $\frac{n+1}d-1> \frac {n-1}2-\frac 76$ it follows that  $(3(n-1)-1)(d-2)<14$. 

If $n=3$,  this implies $d\le 4$. If  $n\geq 4$, this implies  that $3(n-1)-1> 8$, so $d=3$ is the only solution, and therefore $3(n-1)-1< 14$, thus $n \leq 5$. 
\end{proof}

We continue the proof of Theorem \ref{t:pol2} with the remaining cases after the reduction by Lemma \ref{l:1}.

 \smallskip
\noindent
\emph{The case $n=5,d=3$. } 
The spectrum of $f_{5,3}$ is  $\left( \frac 23\right)  + 5\left( 1\right) +10\left( \frac 43\right) +10\left( \frac 53\right) +5\left( 2\right) +\left( \frac 73\right)$, therefore its degree with respect to the open interval $]1,2[$ is
 $20$.  For our cubic fourfold $V$ with finite singular set $\{p_1,\ldots,p_s\}$ and $\pol(V)=2$ we have $\sum_i\mu(V,p_i)= 2^{5}-2 =30$. 
 
The semicontinuity of the spectrum requires that $\sum_i \deg_{]1,2[}\mathrm{Sp}(f_i,p_i) \leq 20$,
thus using the symmetry of the spectrum  with respect to $3/2$ we get:
\[ \sum_i \deg_{]-\infty,1]}\mathrm{Sp}(f_i,p_i) \geq \frac 12(30-20)=5.
\]
But the semicontinuity also gives:
 \[ \sum_i \deg_{]-\infty,1[}\mathrm{Sp}(f_i,p_i) \leq  \deg_{]-\infty,1[}\mathrm{Sp}(f_{5,3}) =1.\]
It follows that the spectral number $1$ in $\sum_i \mathrm{Sp}(f_i,p_i)$ occurs with multiplicity $n_1\geq 4$.  One checks easily that there do not exist singularities in our list satisfying $\sum_i \mu(f_i,p_i)\leq 30$ and $n_1\geq 4$.  Indeed, the only singularities from our list which have a spectral number equal to 1 are $J_{2,0}$ and $J_{4,0}$, with Milnor numbers 10 and 22, respectively,  thus the total multiplicity $n_{1}$ cannot be more than 3.

\medskip

\noindent
\emph{The case $n=4$, $d=3$.} 
The spectrum of $f_{4,3}$ is $ \left(\frac 13\right) +4\left(\frac 23\right) +6\left( 1\right) +4\left(\frac 43\right) +\left(\frac 53\right)$.
For all curve singularity germs $g$ in our table one checks that their smallest spectral number is greater than $-\frac 23$ and 
\[ \deg_{]-\infty,-\frac 13]}\mathrm{Sp}(g,0)\leq \mu(g,0)/4.\]
Hence for a cubic threefold $V$ of polar degree $2$ with singular set $\{p_1,\ldots,p_s\}$, considering the shift of $+1$ of the spectral numbers,  we have: 
\[\sum_i \deg_{]-\infty,\frac 23]}\mathrm{Sp}(f_i,p_{i})\leq \sum_i\mu(f_i, p_{i})/4=\frac{7}{2}.\] 
Using that the smallest spectral number of each possible singularity $(f_i,p_{i})$ is greater than $\frac 13$, and hence by symmetry the greatest spectral number  is smaller than $\frac 53$, we then get:
\[\sum_i \deg_{]\frac 23,\frac 53[}\mathrm{Sp}(f_i, p_{i})=\sum_i \deg_{]\frac 23,\infty[}\mathrm{Sp}(f_i,p_{i})\geq 14-\frac 72>10\]
whereas $\deg_{]\frac 23,\frac 53[}\mathrm{Sp}(f_{4,3})=10$. This contradicts the semicontinuity of the spectrum. 
 
\end{proof}

\subsection{Proof of Theorem \ref{t:mainPol2}}
By the above Theorem \ref{t:pol2} and Lemma \ref{l:1}   we have reduced the proof of Theorem \ref{t:mainPol2} to the cases $n\le 3$ and $d\le 4$.
The projective surfaces of degree $d=3$ have been classified by  Bruce and Wall \cite{BW} and, as noticed by Huh \cite[proof of Prop. 21]{Huh}, the case $\pol(V) =2$ can be extracted  and yields a part of Huh's list. 

The proof for quartic surfaces is also due to Huh \cite[proof of Prop. 21]{Huh}, it is nontrivial and it  uses a whole bunch of classical results.  An alternate proof can be made by using semicontinuity. 

 The case $n=2$ of plane curves with $\pol(V) = d$ has been treated in \cite{FM}, as we have mentioned in the Introduction, and yields the corresponding part of Huh's list.  
 
 Let us remark also that the case $d=2$ and general $n$ is excluded from the list since $\pol(V)$ can be at most 1, 
by formula \eqref{eq:polmu}.

This completes the proof of our Theorem \ref{t:mainPol2}.

\section{Huh's general conjecture: proof of Theorem \ref{t:Huh-compact}}\label{s:general}

For polar degree $k>2$ we are able to find 
bounds for the dimension $n-1$ and degree $d$ of $V$. Our Theorem \ref{t:Huh-compact} answers Huh's Conjecture \ref{c:huh2}  in the more concrete terms of Theorem \ref{t:h0} and \ref{t:Huh2} below.

\subsection{Corank and spectrum}

\begin{theorem}\label{t:h0} 
Let $k\ge 3$. There is no projective hypersurface $V\subset \bP^{n}$ of polar
 degree $k$ with only isolated singular points, for  $n\geq  \max\{ k,5 + 3\log_{2}k\}$   and degree $d\ge 3$. 
\end{theorem}
\begin{proof}
Let $\{p_1,\ldots,p_s\}$ be the set of singular points of a projective hypersurface $V$ of $\mathbb{P}^n$ of polar degree $k\geq 3$. 
By \cite[Theorem 2]{Huh}, one has  $k\ge \mu^{(n-2)}(V,p_{i})$. This is the Milnor number of the slice germ $(V\cap H,p_{i})$ for some general hyperplane $H$ through $p_{i}$. 
If $f_{i}$ is a local equation of  $V$ at $p_{i}$, then the restriction  ${f_{i}}_{|H}$ is a local equation 
for $(V\cap H,p_{i})$. 

Let us denote $r_{i}:= \corank {f_{i}}_{|H}$, where $0\le r_{i}\le n-1$. 
We first show:
\begin{lemma}\label{l:corankbound}
For any $i\in \{1, \ldots , s\}$, we have:
\begin{enumerate}
\rm \item \it $r_{i}\leq \log_{2}k$.
\rm \item \it  The smallest spectral number $\alpha_{1,i}$ of the hypersurface singularity $(V,p_{i})$ is bounded from below by 
 $\frac{n-r_{i}-3}{2}$.
 \end{enumerate}
\end{lemma} 
\begin{proof}
We shall use the following two standard facts concerning the corank. Let $h:(\bC^{m},0)\to (\bC,0)$ be some function germ with isolated singularity. 

\noindent
(1).  If $r = \corank h$   then $h =  g(x_{1},  \ldots , x_{r})  + x^{2}_{r+1}+\cdots + x^{2}_{m}$ with $\ord_{0} g \ge 3$  and therefore (by using the Sebastiani-Thom formula) we get  $\mu(h) = \mu( g) \ge \mu(x^{3}_{1}+\cdots + x^{3}_{r}) = 2^{r}$.
  
\noindent
(2). If the restriction $h_{|\{x_{1}=0\}}$ has isolated singularity and has corank $p$,  then $\corank h \le p+1$. Indeed one has $h_{|\{x_{1}=0\}} = \hat h(x_{2}, \ldots , x_{p+1}) + x^{2}_{p+2}+\cdots + x^{2}_{m}$ and thus \\
$h = x_{1}h_{1}(x_{1}, \ldots , x_{p+1}) + x_{1}h_{2}(x_{p+2}, \ldots, x_{m})  +   x^{2}_{p+2}+\cdots + x^{2}_{m}$ is right-equivalent, modulo $\frak{m}^{3}$, to  $x_{1}\tilde h_{1}(x_{1}, \ldots , x_{p+1})  +   x^{2}_{p+2}+\cdots + x^{2}_{m}$  where    $x_{1}\tilde h_{1}(x_{1}, \ldots , x_{p+1}) $ is either of order $\ge 3$ or it is the sum of $x_{1}^{2}$ and some function of order $\ge 3$.

\medskip

\noindent
 (a).  Using fact (1),  we get that the Milnor number  $\mu({f_{i}}_{|H}, p_{i})$ is bounded from below by $2^{r_{i}}$. 
 Then, from the above inequalities,  we get:
\[  r_{i}\le \log_{2}\mu^{(n-2)}(V,p_{i})\leq \log_{2}k .
\]

\noindent
(b). By fact (2),  after slicing with any hyperplane such that the local singularity is still an isolated singularity of the function restricted to the slice, the corank can drop by at most one. We thus have:
  \[ \corank f_{i} \le r_{i}+1 \le 1+ \log_{2}\mu^{(n-2)}(V,p_{i})\leq 1+ \log_{2}k .
\]
  Therefore 
we may write in some well chosen system of local coordinates:
\[ f_{i}(x_{1}, \ldots, x_{n}) = g(x_{1}, \ldots, x_{r_{i}+1}) + x_{r_{i}+2}^{2} + \cdots + x_{n}^{2},
\] 
where $g\in m^{3}$ is a germ of a function in $r_{i}+1$ variables with trivial $2$-jet.  By using the properties of the spectrum, we get that the spectrum of $g$ is in the interval $]-1, r_{i}[$ and that by adding $n-r_{i}-1$ squares  to $g$, since the spectrum shifts by  $+\frac12$ for each square, the spectrum of $f_{i}$ is in the interval:
\[   \left]-1 +\frac{n-r_{i}-1}{2}, r_{i} +\frac{n-r_{i}-1}{2}\right[ .
\] 
and we have  $n-r_{i}-1 >0$ for any $i$, since $r_{i}\leq \log_{2}k \le \log_{2}n < n-1$ by our hypotheses $n\ge k\ge 3$. 
 \end{proof}

 If the second spectral number of $f_{n,d}$, which is $\frac{n+1}d-1$ and has multiplicity $n$, is smaller than $\beta := \min_{i}\frac{n-r_{i}-3}{2}$, then:
\[ \deg_{[\beta,\infty[}\mathrm{Sp}(f_{n,d}) \leq (d-1)^n -n-1.\]
On the other hand we have:
\[\deg_{[\beta,\infty[}\sum_i\mathrm{Sp}(f_{i},p_i) =\sum_i\mu(f_{i,} p_{i})=(d-1)^n-k.\]
By the semicontinuity of the spectrum we get:
\[(d-1)^n-k\leq (d-1)^n-n-1, \]
or $n+1\leq k$, which contradicts the hypothesis $n\geq k$. 

We therefore must have $-1 +\frac{n+1}d\geq \beta$, and since $r_{i}\le \log_{2}k$ for any $i$, we get $\beta\ge \frac{n-\log_{2}k-3}{2}$. This amounts to the inequality: 
\[d\leq \frac{2(n+1)}{n-1-\log_{2}k}. \]
Since the last fraction is less than $3$ if $n>5 + 3\log_{2}k$, our proof of Theorem \ref{t:h0} is complete.
\end{proof}


\subsection{An upper bound for the degree $d$}
 The first spectral number of $g$ at the point $x$ will be denoted by $\alpha_{1}(g,x)$. We continue to use the notation $f_{n,d}$ for a generic homogeneous polynomial of degree $d$ in $n$ variables, which is deformation-equivalent to $\sum_{i=1}^{n} x_i^d$.

\begin{lemma}\label{l:ineq}
Let $\pol(V) = k\ge 2$, where $V = \{ f=0\} \subset \bP^{n}$ has isolated singularities. Then:
\[ \alpha_{1}(f,p)  > -1 + \frac{n-1}{k+2}
\] 
for any $p\in \Sing V := \{p_{1}, \ldots , p_{s}\}$.
\end{lemma}
\begin{proof}
One has Huh's inequality  \cite[Theorem 2]{Huh}:  $\mu^{(n-2)}(V,p_{i})\leq k$, for any $i$. 
In the notations of the proof of Theorem \ref{t:h0}, this amounts to $\mu(f_{i|H}, p_{i}) \leq k$, where $H$ is some general hyperplane through $p_{i}$. We consider the composition of $f_{i|H}$ with the translate of the point $p_i$ at the origin and denote the result by  $\tilde f_{i|H}$.
 It follows that the function germ $\tilde f_{i|H}$ is $k+1$-determined, which means that it is deformation-equivalent 
 to its $(k+1)$-jet $g_{i}:= \j^{k+1}(\tilde f_{i|H})$, which is a polynomial of degree $k+1$. We then apply the spectrum semi-continuity to the deformation:
\[   h_{i} := s g_{i} +(1-s) f_{n-1, k+2}\]
and get $\alpha_{1}(f_{i|H}, p_{i}) = \alpha_{1}(g_{i}, 0) \geq \alpha_{1}(f_{n-1, k+2},0) = -1 + \frac{n-1}{k+2}$.
By the next Lemma \ref{l:h2} one has $\alpha_{1}(f_{i|H}, p_{i}) < \alpha_{1}(f_{i|}, p_{i})$, and our claim follows by chaining these inequalities.
\end{proof}

\begin{lemma}\label{l:h2} 
Let $h:(\mathbb{C}^n,0)\to (\mathbb{C},0)$ be an isolated hypersurface singularity with smallest spectral number $\alpha_1$. Let $H$ be a general hyperplane through $0$ in $\mathbb{C}^n$ and let $\alpha_1^\prime$ be the smallest spectral number of the restriction $h_{|H}$. Then $\alpha_1> \alpha^\prime_1$. 
 \end{lemma}
\begin{proof} Since $h$ is finitely determined,  we may assume without loss of generality that $h$ is a polynomial, that $H$ is defined by $z_1=0$ and that 
\[ h(z_1,\ldots,z_n) = \sum_{m=0}^{N-1}z_1^mh_m(z_2,\ldots,z_n)+z_1^N\]
with $\deg h_m \le N-m$ and $N \gg 1$. Then by: 
\[ h_t(z_1,\ldots,z_n):= \sum_{m=0}^{N-1}(tz_1)^mh_m(z_2,\ldots,z_n)+z_1^N\]
we obtain $h$ as a deformation of the singularity $z_1^N+h_0(z_2,\ldots,z_n)$. By the Thom-Sebastiani theorem for the spectrum, this singularity has its smallest spectral number equal to $\alpha_1^\prime+\frac 1N$ and by the semicontinuity of the spectrum we get:  
$\alpha_1\geq \alpha_1^\prime+\frac 1N>\alpha_1^\prime$.
\end{proof}

\smallskip

Let $n,k\geq 2$ be fixed. Let $v_j$ denote the multiplicity of the $(j+1)$st spectral number $-1+\frac{n+j}d$ of the germ $f_{n,d}$,  for some $d\geq 2$. In particular $v_0=1$ and $v_1=n$. In general these numbers depend both on $n$ and $d$. However, for $j\leq d-2$ we have $v_j={n+j-1\choose n-1}$, see the general formula for the spectrum of $f_{n,d}$ in \S \ref{ss:gen}. We  define 
\begin{equation}\label{eq:ell}
  \ell_{n,k} := \min \left\{\ell \in \bN \ \middle| {n+\ell \choose n}>k\right\} .
  \end{equation}

  The  following estimation  of the  upper bound for the degree $d$ completes now the proof of Theorem \ref{t:Huh-compact}:

\begin{theorem}\label{t:Huh2}
Let  $n,k\geq 2$.  Any hypersurface $V\subset \bP^n$ with isolated singularities and $\pol(V) = k$ has degree  
\[ d < \max \{ 2+ \ell_{n,k},  (n+\ell_{n,k})(k+2)/(n-1)\}.\]
\end{theorem}
\begin{proof}
In case  $d\geq 2 + \ell_{n,k}$, by the definition \eqref{eq:ell}, the expression for $v_{j}$ given above,
and by taking into account the identity $\sum_{j=0}^\ell {n+j-1\choose n-1} ={n+\ell \choose n}$ for any $\ell \ge 0$, we have: 
\begin{equation}\label{eq:ell2}
\sum_{j=0}^{\ell_{n,k}} v_{j}= {n+\ell_{n,k} \choose n} >k. 
  \end{equation}

Let us then denote $\gamma_{n,k}:= -1 + \frac{n-1}{k+2}$. We will show that $\gamma_{n,k} < -1 + \frac{n+\ell_{n,k}}{d}$, as follows. 

 Assume by contradiction that $\gamma_{n,k} \ge -1 + \frac{n+\ell_{n,k}}{d}$.
Since $V$ has singular points $p_1,\ldots,p_s$ with total Milnor number $\mu=\sum_{i=1}^s\mu_i$,  by Lemma \ref{l:ineq} we get: 
\[\deg_{]\gamma_{n,k},\infty[}\sum_i\mathrm{Sp}(f_{i},p_i) =\mu=(d-1)^n-k.\]

On the other hand, by our assumption we have:
\[ \deg_{]\gamma_{n,k},\infty[}\mathrm{Sp}(f_{n,d}) \leq (d-1)^n - (1 +  v_1 +\ldots + v_{\ell_{n,k}}).\]

Applying the semicontinuity of the spectrum we get the inequality:
\[ (d-1)^n-k \le (d-1)^n - (1 +  v_1 +\ldots + v_{\ell_{n,k}})\]
which contradicts \eqref{eq:ell2}.

 We have thus shown $\gamma_{n,k} < -1 + \frac{n+\ell_{n,k}}{d}$. This amounts to the inequality
 $d< (n+\ell_{n,k})(k+2)/(n-1)$, which concludes the proof of our theorem.
\end{proof}


\end{document}